\documentclass[a4paper, 12pt]{article}

\usepackage[margin=2cm]{geometry}
\usepackage{hyperref,xcolor}
\usepackage{amsthm, amssymb, amsmath}
\usepackage{enumerate}
\usepackage{algorithmic}
\usepackage{authblk}
\usepackage{graphics}
\usepackage{pstricks}
\usepackage{subfigure}

\hypersetup{
    colorlinks=false,
	linkcolor=red,
	citecolor=red,
    pdfborder={0 0 0},
}

\title{Partial list colouring of certain graphs}

\author[1]{Jeannette Janssen}
\author[2]{Rogers~Mathew\footnote{Supported by an AARMS Postdoctoral Fellowship}}
\author[2]{Deepak Rajendraprasad}

\affil[1]
{
	Department of Mathematics and Statistics, \authorcr 
	Dalhousie University, 
	Halifax, Canada - B3H 3J5. \authorcr
	janssen@mathstat.dal.ca
}
\affil[2]
{
	Department of Computer Science, \authorcr 
        Caesarea Rothschild Institute, \authorcr 
        University of Haifa, 31905 Haifa, Israel \authorcr
        \{rogersmathew,deepakmail\}@gmail.com
}

\theoremstyle{definition}

\theoremstyle{plain}
\newtheorem{theorem}{Theorem}
\newtheorem{lemma}[theorem]{Lemma}
\newtheorem{corollary}[theorem]{Corollary}

\newtheorem{proposition}[theorem]{Proposition}

\theoremstyle{remark}

\newtheoremstyle{plainitshape}% name of the style to be used
  {}% measure of space to leave above the theorem. E.g.: 3pt
  {}% measure of space to leave below the theorem. E.g.: 3pt
  {\itshape}% name of font to use in the body of the theorem
  {}% measure of space to indent
  {\itshape}% name of head font
  {.}% punctuation between head and body
  {0.5em}% space after theorem head
  {}% Manually specify head
\theoremstyle{plainitshape}

\newtheoremstyle{cases}% name of the style to be used
  {}% measure of space to leave above the theorem. E.g.: 3pt
  {}% measure of space to leave below the theorem. E.g.: 3pt
  {}% name of font to use in the body of the theorem
  {}% measure of space to indent
  {}% name of head font
  {\newline}% punctuation between head and body
  {0.5em}% space after theorem head
  {{\itshape \thmname{#1}} \thmnumber{#2} ({\itshape\thmnote{#3}}).\medskip}% Manually specify head
\theoremstyle{cases}

\newtheoremstyle{constructions}% name of the style to be used
  {}% measure of space to leave above the theorem. E.g.: 3pt
  {}% measure of space to leave below the theorem. E.g.: 3pt
  {}% name of font to use in the body of the theorem
  {}% measure of space to indent
  {}% name of head font
  {}% punctuation between head and body
  {0.5em}% space after theorem head
  {{\itshape \thmname{#1}} \thmnumber{#2}\medskip}% Manually specify head
\theoremstyle{constructions}

\begin{document}
\maketitle
\begin{abstract}
Let $G$ be a graph on $n$ vertices and let $\mathcal{L}_k$ be an arbitrary function that assigns each vertex in $G$ a list of $k$ colours. Then $G$ is \emph{$\mathcal{L}_k$-list colourable} if there exists a proper colouring of the vertices of $G$ such that every vertex is coloured with a colour from its own list. We say $G$ is \emph{$k$-choosable} if for every such function $\mathcal{L}_k$, $G$ is $\mathcal{L}_k$-list colourable. The minimum $k$ such that $G$ is $k$-choosable is called the \emph{list chromatic number} of $G$ and is denoted by $\chi_L(G)$. Let $\chi_L(G) = s$ and let $t$ be a positive integer less than $s$. 
%Consider an arbitrary function $\mathcal{L}_t$ that maps each vertex in $G$ to a list of $t$ colours. T
The \emph{partial list colouring conjecture} due to Albertson et al. \cite{albertson2000partial} states that for every $\mathcal{L}_t$ that maps the vertices of $G$ to $t$-sized lists, there always exists an induced subgraph of $G$ of size at least $\frac{tn}{s}$ that is $\mathcal{L}_t$-list colourable. 
%What if one assigns only a $t$-sized list at each vertex of $G$, where $t$ is some positive integer less than $s$?  The partial list colouring conjecture, raised by Albertson et al. in \cite{albertson2000partial}, states that one can still properly colour at  least $\frac{tn}{s}$ vertices of $G$ such that no vertex is given a colour which is not in its list. 
In this paper we show that the partial list colouring conjecture holds true for certain  classes of graphs like claw-free graphs, graphs with large chromatic number, chordless graphs, and series-parallel graphs. 

In the second part of the paper, we put forth a question which is a variant of the partial list colouring conjecture: does $G$ always contain an induced subgraph of size at least $\frac{tn}{s}$ that is $t$-choosable? 
%Given a graph $G$ on $n$ vertices with list chromatic number $s$, for any positive integer $t<s$, does $G$ contain a $t$-choosable graph of size $\frac{tn}{s}$?  In the second part of the paper we address a slightly different question: does $G$ always contains a $t$-choosable (induced) subgraph of size at least $\frac{tn}{s}$?  
%This question clearly is a strengthening of the partial list colouring conjecture due to Albertson et al. 
We show that the answer to this question is not always `yes' by explicitly constructing an infinite family of $3$-choosable graphs where a largest induced $2$-choosable subgraph of each graph in the family is of size at most $\frac{5n}{8}$. 
%which does not contain any $2$-choosable graph of size $\frac{2n}{3}$. 
\end{abstract}
\section{Introduction}
% Let $A$ be any set. For a positive integer $b$, where $ b \leq |A|$, we shall use $A \choose b$ to denote the collection of all $b$-sized subsets of $A$. 

Consider a simple, undirected, and finite graph $G$. Let $\mathcal{L} = \{l(v) : v \in V(G)\}$ denote an assignment of a list of admissible colours for each vertex in $G$. Then $G$ is \emph{$\mathcal{L}$- list colourable} if there exists a proper colouring (i.e., no two adjacent vertices get the same colour) of the vertices of $G$ such that each vertex $v$ is assigned a colour from $l(v)$. If $\forall v \in V(G), |l(v)| = k$ then $\mathcal{L}$ is called an \emph{$k$-assignment}.  We say $G$ is \emph{$k$-choosable} if $G$ is $\mathcal{L}$-list colourable for all $k$-assignments $\mathcal{L}$.  The \emph{list chromatic number} of $G$, denoted by $\chi_L(G)$, is the minimum positive integer $k$ such that $G$ is $k$-choosable.  
%Albertson et al. in \ref{} raised the following question: 

Consider a graph whose chromatic number is $s$. Take an $s$-colouring. We know that, for every positive integer $t$ less than $s$, one can properly colour at least $\frac{tn}{s}$ of its vertices using only $t$ colours  (by taking the vertices of the largest $t$ colour classes of  the $s$-colouring). Albertson, Grossman, and Haas in \cite{albertson2000partial} asked a similar and a very natural question on list colouring. Consider a  graph $G$ on $n$ vertices whose list chromatic number is $s$. What if one assigns lists of size $t$ to each vertex where $t$ is some positive integer less than $s$? We know that with such a list assignment it may not be possible to list colour all the vertices of $G$. But the question is not to colour all the vertices but to colour as many vertices as one can. They conjectured that, given a $t$-assignment, one can always find an induced subgraph of size at least $\frac{tn}{s}$ that can be properly list coloured. We give a more formal description of the conjecture below: 
\\
\textit{Partial list colouring conjecture} (Conjecture $1$ in \cite{albertson2000partial}): Consider an arbitrary graph $G$ on $n$ vertices whose list chromatic number is $s$. Let $t$ be any positive integer less than $s$. Let $\mathcal{L}_t = \{l_t(v) : v \in V(G)\}$ be any $t$-assignment for $G$. Let $\lambda_{\mathcal{L}_t}(G)$ denote the size of a largest induced subgraph of $G$ that is $\mathcal{L}_t$-list colourable. Let $\lambda_t(G) = \min\{\lambda_{\mathcal{L}_t}(G): \mathcal{L}_t \mbox{ is a }t\mbox{-assignment for }G\}$. The partial list colouring conjecture states that $\lambda_t(G) \geq \frac{tn}{s}$. 

The authors in \cite{albertson2000partial} showed that $\lambda_t(G)$ is always at least $\left(1-(\frac{\chi(G)-1}{\chi(G)})^t\right)n$, where $\chi(G)$ denotes the chromatic number of $G$. For a bipartite graph this proves the conjecture as by induction one can see that $\left(1-(\frac{\chi(G)-1}{\chi(G)})^t\right)n \geq \frac{tn}{\chi(G) + t -1} = \frac{tn}{t+1}$. Chappell in \cite{chappell1999lower} showed that $\lambda_t(G)$ is lower bounded by $\frac{6}{7}\frac{tn}{s}$. Janssen in 
\cite{janssen2001partial} proved that the conjecture holds true for every graph whose list chromatic number is at least its maximum degree. In \cite{haas2003bounds} it was shown that $\lambda_t(G) \geq \frac{n}{\lceil\frac{s}{t}\rceil}$. 
\subsection{Notations and Definitions}
For any $S \subseteq V(G)$, we use $G \setminus S$ to denote the subgraph of a graph $G$ induced on the vertex set $V(G) \setminus S$. For a $v \in V(G)$, we use $G\setminus v$ to denote the graph $G\setminus \{v\}$. Let $N_G(S) := \{v \in V(G)\setminus S~|~v\mbox{ has a neighbour in }  S\}$. For a $v \in V(G)$, we use $N_G(v)$ to denote $N_G(\{v\})$. Let $deg_G(v) := |N_G(v)|$. For any positive integer $n$, we use $[n]$ to denote the set $\{1, \ldots , n\}$. 

A graph $G$ is \emph{$k$-degenerate} if the vertices of $G$ can be arranged on a horizontal line from left to right such that no vertex has more than $k$ neighbours to its right. From such an ordering of the vertices it's easy to see that if $G$ is $k$-degenerate then $\chi(G) \leq k+1$ and $\chi_L(G) \leq k + 1$. 
\subsection{Outline of the paper}
In the subsections of Section \ref{ProofSection} we prove that the partial list colouring conjecture holds true for various graphs classes like claw-free graphs, graphs of large chromatic number, and for some sub-families of $3$-choosable graphs like chordless graphs and series-parallel graphs. In Section \ref{StrongerConjectureSection}, we look at a question that is a variant of the partial list colouring conjecture: does an $s$-choosable graph $G$ on $n$ vertices always contain an induced subgraph of size at least $\frac{tn}{s}$ that is $t$-choosable? We show that the answer to this question is not always `yes' by explicitly constructing an infinite family of $3$-choosable graphs where a largest induced $2$-choosable subgraph of each graph in the family is of size at most $\frac{5n}{8}$.     
%The chromatic number of $G$ is denoted by $\chi(G)$. 
% \section{Notations and definitions}
%\section{Some useful propositions}
%\section{Proving the conjecture on certain graph families}
\section{Partial list colouring of some graphs}
\label{ProofSection}

\subsection{Claw-free graphs}
% \begin{proof}
% Consider a graph $G \in \mathcal{G}$ with $n$ vertices having a list chromatic number of $s$. Let $t$ be any integer such that $0 \leq t \leq s-1$. 
% We prove this by an induction on $t$. The case when $t=s-1$ forms the base case which is given to be true. Assume the statement of the lemma is true for every $t>r$, where $0 < r < s$. Consider the case when $t=r$.  Let $R = \{c_1, \ldots , c_p\}$ be any set of colours such that $|R| \geq r$. Consider an arbitrary list assignment, say $l': V(G) \rightarrow {R \choose r}$, that assigns an $r$-sized list from $R$ to each vertex in $G$. For each $v \in V(G)$, we assign a list $l(v)$ which is defined as $l(v) = l'(v) \cup \{c_{\infty}\}$, where $c_{\infty} \notin R$. By induction hypothesis, there exists an induced subgraph $G_{r+1}$ of $G$ such that $|V(G_{r+1})| \geq \frac{(r+1)n}{s}$ and $G_{r+1}$ has a proper vertex colouring where each vertex $v$ is assigned a colour from $l(v)$.      
% \end{proof}
Here we prove that the partial list colouring conjecture holds true for all claw-free graphs. The proof technique is similar to the one used  in \cite{janssen2001partial} to prove the conjecture for a graph whose list chromatic number is at least its maximum degree. Given a list assignment for a graph where the size of lists is smaller than required, we append every list with a set  of new colours such that the new list size equals the list chromatic number of the graph. Next we do a proper list colouring of the graph using the new lists such that the number of vertices that are assigned a new colour is minimized. A simple counting argument is then used to prove that the above list colouring colours sufficient number of vertices using a colour from their original lists.  
\begin{theorem}
Let $G$ be a claw-free graph on $n$ vertices whose list chromatic number is $s$. Then for every $t$, where $0 < t <  s$, $\lambda_t(G) \geq \frac{tn}{s}$. 
\end{theorem}
\begin{proof}
%From Proposition \ref{InductiveRemark} we know that in order to prove the theorem it is enough to prove that every graph in the family of claw-free graphs satisfy Property $\mathcal{P}$. 
Let $\mathcal{L}_t = \{l_t(v)~:~v \in V(G)\}$ be an arbitrary $t$-assignment for $G$. Let $\bigcup_{v \in V(G)} l_t(v) = \{1, \ldots , p\}$. Let $\mathcal{L}_{s} = \{l_s(v)~:~l_s(v) = l_t(v) \cup \{\sigma_1, \ldots , \sigma_{s-t}\}, v \in V(G)\}$.  Since $\chi_L(G) = s$, $G$ is $\mathcal{L}_s$-list colourable. 
%In order to prove that $G$ satisfies property $P$ we have to show that there exists an induced subgraph of $G$ of size at least $\frac{(s-1)n}{s}$ that is $l_{s-1}$-colourable. 
Let $f:V(G) \rightarrow \{1, \ldots, p, \sigma_1, \ldots , \sigma_{s-t}\}$ be an $\mathcal{L}_s$-list colouring of $G$ such that the number of vertices that receive a colour from the set  $\{\sigma_1, \ldots , \sigma_{s-t}\}$ is minimum. For each $i \in \{1, \ldots p, \sigma_1, \ldots , \sigma_{s-t} \}$, let $C_i := \{v \in V(G)~|~f(v) = i\}$ and let $V_i := \{v \in V(G)~|~i \in l_s(v)\}$.  If $|C_{\sigma_1} \cup \cdots \cup C_{\sigma_{s-t}}| \leq \frac{(s-t)n}{s}$, then the theorem is proved. Suppose $|C_{\sigma_1} \cup \cdots \cup C_{\sigma_{s-t}}| > \frac{(s-t)n}{s}$. Then there exists some $i \in \{1, \ldots , s-t\}$ such that $|C_{\sigma_i}| > \frac{n}{s}$. For ease of notation, from now  we shall use $\sigma$ to denote one such  $\sigma_i$. Then,  
\begin{eqnarray}
\label{ineq1}
\sum_{i=1}^{p}|C_i| & < & \frac{tn}{s}
\end{eqnarray} and 
\begin{eqnarray}
\label{ineq2}
\sum_{i=1}^{p}|C_{\sigma} \cap V_i| > \frac{tn}{s}. 
\end{eqnarray}
From Inequalities \ref{ineq1} and \ref{ineq2}, we can conclude that there exists some $j \in [p]$ such that $|C_j| <  |C_{\sigma} \cap V_j|$. Find an $X \subseteq C_{\sigma} \cap V_j$ such that its neighbourhood in $C_j$ is minimum and is of a size smaller than $|X|$. 
%$|N_G(X) \cap C_j|$ is minimum and $|N_G(X) \cap C_j| < X$. 
The existence of such an $X$ is guaranteed by the presence of the set $C_{\sigma} \cap V_j$ whose neighbourhood in $C_j$ is smaller than itself in size. We claim that every vertex in $N_G(X) \cap C_j$ has at least two neighbours in $X$. Otherwise, if a vertex in the neighbourhood of $X$ in $C_j$ has only one or no neighbour in $X$ then we can remove such a vertex only to reduce the size of both $N_G(X) \cap C_j$ and $X$ by at most $1$. But this contradicts the minimality of $|N_G(X) \cap C_j|$. Hence every vertex in $N_G(X) \cap C_j$ has at least two neighbours in $X$. Since $G$ is claw-free and  every vertex in $N_G(X) \cap C_j$ is having at least two neighbours in $X$, no vertex in $N_G(X) \cap C_j$ has a neighbour in $C_{\sigma} \setminus X$. By assigning colour $j$ to vertices in $X$ and colour $\sigma$ to vertices in $N_G(X) \cap C_j$, one can find another valid $\mathcal{L}_s$-list colouring of $G$ that has lesser number of vertices taking colours $\sigma_1, \ldots , \sigma_{s-t}$  as compared to $f$. This contradicts the property of $f$.
% is an $\mathcal{L}_s$-colouring of $G$ that has the minimum number of vertices getting colours  $\sigma_1, \ldots , \sigma_{s-t}$. 
%Hence, our assumption that $|C_{\sigma}| > \frac{n}{s}$ is wrong. 
%    Thus, one can find another $l_s$ colouring of $G$ defined as: $f'(v) $f'(v) = f(v)$, otherwise. 
%We shall prove the theorem for $t=s-1$ by showing that we can do a proper vertex colouring of and induced subgraph $H$
% 5 of $G$ such that $|V(H)| \geq \frac{(s-1)n}{s}$ and every vertex $v$ in $H$ gets a colour from $l'(v)$.  Let 
% 5$G_{r+1}$ be the induced subgraph of $G$ such that $|G_{r+1}| \geq \frac{(r+1)n}{s}$.
\end{proof}
Since line graphs are claw-free graphs, we have the following corollary.  
\begin{corollary}
Let $G$ be a line graph (of any multigraph) on $n$ vertices whose list chromatic number is $s$. Then for every positive integer $t$ less than $s$, $\lambda_t(G) \geq \frac{tn}{s}$. 
\end{corollary}

\subsection{Graphs of large chromatic number}
Ohba in \cite{OhbaConjecture} conjectured that, for a graph $G$ on $n$ vertices, if $\chi(G) \geq \frac{n-1}{2}$ then $\chi(G) = \chi_L(G)$. In 2012, the conjecture was settled in the affirmative by Noel, Reed, and Wu \cite{NoelReedWu}. 
We use this result to prove the following theorem. 
\begin{theorem}
\label{largeChromatThm}
Let $G$ be a graph on $n$ vertices whose list chromatic number is $s$. If $\chi(G) \geq \frac{n-1}{2}$ then for every positive integer $t$ less than $s$, $\lambda_t(G) \geq \frac{tn}{s}$.   
\end{theorem}
\begin{proof}
%We know that $\chi_L(G) = \chi(G) = s$. Consider a proper vertex colouring of $G$. Let $G'$ be the subgraph induced on the vertices of the smallest $t$ colour classes of this colouring. Let $n' = |V(G')|$. Since $n' \leq \frac{tn}{s}$ and $n < 2s+1$, we have $\frac{n'-1}{2} \leq \frac{tn - s}{2s} < \frac{t(2s+1)-s}{2s} = \frac{(2t-1)s + t}{2s} = t - \frac{1}{2} + \frac{t}{2s} \leq t - \frac{1}{2} + \frac{s-1}{2s} \mbox{( since }t \mbox{ is at most } s-1) < t$. Thus $\chi(G') = t > \frac{n'-1}{2}$. Hence, by Theorem [] in \ref{}, 
We shall prove the theorem when $t=s-1$. Extending the same argument gives a proof for any $t$. The idea of the proof is to find an induced subgraph $G'$ of $G$ such that $\chi_L(G') = s-1$ and $|V(G')| \geq \frac{(s-1)n}{s}$. 
%If $G$ is a complete graph, then we know that the statement of the theorem is trivially true (it also follows from Observation \ref{}). Suppose $G$ is not a complete graph. 

Since $\chi(G) \geq \frac{n-1}{2}$, from the result in \cite{NoelReedWu}, we know that $\chi_L(G) = \chi(G) = s$. Consider a proper vertex colouring of $G$ using $s$ colours. Let $C_1, \ldots , C_s$ be the $s$ colour classes of this colouring. If $s > \frac{n}{2}$, then there exists some $C_i$ of size $1$. Let $G' = G \setminus C_i$. We then have $\chi(G') = s-1 > \frac{n}{2} - 1 \geq \frac{|V(G')| - 1}{2}$ and therefore, by the result in \cite{NoelReedWu}, $G'$ is $(s-1)$-choosable. Note that $|V(G')| = n-1 \geq \frac{(s-1)n}{s}$, since $n$ is always at least $s$.  
%Consider any $i \in \{1, \ldots, s\}$. 

Suppose $\chi(G) = s \leq \frac{n}{2}$. 
%It is given that $s \geq \frac{n-1}{2}$. 
If, for some $a \in \{1, \ldots ,s\}$, $|C_a| = 2$ then let $G' = G \setminus C_a$. Otherwise, since $\frac{n}{2} \geq \chi(G) \geq \frac{n-1}{2}$ there exist colour classes $C_b, C_c$ such that $|C_b|=1$ and $|C_c|>2$.  Consider a vertex $u \in C_c$. If $\chi(G \setminus u) = s$, then let $G' = G \setminus C_b \cup \{u\}$. Otherwise, if  $\chi(G \setminus u) = s-1$ then let $v$ be a vertex other than $u$ present in $C_c$ and let $G' = G \setminus \{u,v\}$. Now consider the graph $G'$. No matter how $G'$ was constructed from $G$, we have $\chi(G') = s-1$ and $|V(G')| = n-2$. We thus have $\chi(G') = s-1 \geq \frac{n-1}{2} - 1 = \frac{|V(G')| -1}{2}$. Hence, by the result in \cite{NoelReedWu}, $\chi_L(G') = \chi(G') = s-1$. Note that $|V(G')| = n-2 \geq \frac{(s-1)n}{s}$ since $n$ is at least $2s$.

%Let $G'$ be the subgraph induced on the vertices of the largest $s-1$ colour classes of this colouring. Clearly, $|V(G')| \geq \frac{(s-1)n}{s}$ and $|V(G) \setminus V(G')| \geq 1$. That is, $s-1 \geq \frac{n-1}{2}-1 = \frac{n-3}{2} \frac{|V(G')|-1}{2} \leq \frac{n-3}{2} = \frac{n-1}{2} - 1 < s-1$. Hence, by Theorem [] in \ref{}, $\chi_L(G') = s-1$. 
\end{proof}

\subsection{Chordless graphs}
A graph $G$ is \emph{chordless} if no cycle in $G$ has a chord. Further, $G$ is \emph{minimally $2$-connected} if $G$ is $2$-connected and chordless. Any graph obtained from a given graph by subdividing every edge of the given graph at least once is an example of a chordless graph. If the given graph is $2$-connected, then the resultant graph is minimally $2$-connected. Chordless graphs are $2$-degenerate and therefore $3$-choosable. 

The following lemma about minimally $2$-connected graphs is due to Plummer \cite{plummer1968minimal}.
\begin{lemma}
\label{partitionlemma}
Let $G$ be a $2$-connected graph. Then $G$ is minimally $2$-connected if and only if either 
\begin{itemize}
\item $G$ is a cycle; or 
\item if $S$ denotes the set of nodes of degree $2$ in $G$, then there are at least two components in $G\setminus S$ where each component is a tree. 
\end{itemize} 
\end{lemma}
We use this lemma to prove the following. 
\begin{lemma}
\label{degreeLemma}
Let $G$ be a minimally $2$-connected graph. 
% Then for all $x \in V(G)$ there exist two vertices $v$ and $w$, each of degree $2$, in the neighbourhood of some vertex $u$ such that $x \notin \{v,w\}$. 
Then $\forall x \in V(G)$, $\exists u,v,w \in V(G)$ such that $v,w \in N_G(u)$, $deg_G(v) = deg_G(w) = 2$, and $x \notin \{v,w\}$. 
\end{lemma} 
\begin{proof}
Consider any $x \in V(G)$. If $G$ is a cycle, then let $u = x$. Clearly, $u$ has two neighbours of degree $2$ none of which is $x$.  Suppose $G$ is not a cycle. Then, by Lemma \ref{partitionlemma}, $G \setminus S$ is a forest with at least $2$ components, where $S$ is the set  of nodes of degree $2$ in $G$. If $G \setminus S$ has any isolated vertex, then it has at least $3$ neighbours in $S$ (each of degree $2$) as each vertex outside $S$ has a degree at least $3$. Let the isolated vertex be our vertex $u$. Clearly, $u$ has two neighbours each of degree $2$ in its neighbourhood such that none of them is $x$. Suppose $G \setminus S$ does not contain any isolated vertex. We know that $G \setminus S$ contains at least two trees, say $T_1$ and $T_2$, by Lemma \ref{partitionlemma}. Let $l_1^1$ and $l_2^1$ be two leaf vertices of $T_1$. Similarly, let $l_1^2$ and $l_2^2$ be two leaf vertices of $T_2$. For all $i,j \in \{1,2\}$, since $l_i^j$ has degree at least $3$ in $G$, it has at least $2$ neighbours (each of degree $2$) in $S$. What is left is to show that, for some $i,j \in \{1,2\}$, $l_i^j$ ($=u$) has two neighbours each of degree $2$ in $S$ such that none of the two neighbours is $x$. Since every vertex of $S$ has its degree equal to $2$ in $G$, no three vertices in $G\setminus S$ can have the same neighbour  in $S$. Since $l_i^j$ form a collection of four vertices, there exists one vertex in this collection (which forms the vertex $u$) such that it has two neighbours of degree $2$ none of which is $x$. 
\end{proof}
%\begin{lemma}
%Let $G$ be a chordless graph. Then, there always exist vertices $u,v$ and $w$ in $G$ such that $v,w \in N_G(u)$ and 
%\end{lemma}
\begin{theorem}
Let $G$ be a chordless graph on $n$ vertices whose list chromatic number is $s$. Then for every $t$, where $0 < t <  s$, $\lambda_t(G) \geq \frac{tn}{s}$. 
\end{theorem}
\begin{proof}
% Let $R = \{1, \ldots , p\}$ be any set of colours such that $|R| \geq s$. Consider an arbitrary list assignment, say $l_{s-1}: V(G) \rightarrow {R \choose (s-1)}$, that assigns an $(s-1)$-sized list from $R$ to each vertex in $G$.
Since every chordless graph is $2$-degenerate, $s \leq 3$. It is easy to see that the theorem is trivially true when $s \leq 2$. Assume $s =3$. We know that, when $t=1$, $\lambda_1(G) \geq \frac{n}{3}$ since the largest independent set in $G$ is of size at least $\frac{n}{3}$. Let $t=2$. In rest of the proof, using an induction on $n$, we show that $\lambda_2(G) \geq \frac{2n}{3}$.  
%It is easy to see that for every chordless graph on $1$, $2$, or $3$ vertices, the statement of the theorem holds. 
Suppose $G$ contains a vertex $v$ of degree at most $1$. By induction hypothesis, $\lambda_2(G \setminus v) \geq \frac{2(n-1)}{3}$. Since $v$ has only one neighbour in $G$, we can add back $v$ to $G \setminus v$ and colour $v$ with a colour in its list that does not conflict with its neighbour's colour. Thus, $\lambda_2(G) \geq \frac{2n}{3}$. Suppose every vertex in $G$ is of degree at least $2$. Consider the block graph of $G$ and consider a leaf block $B$ in it. Since $G$ does not have isolated vertices, $B$ is $2$-connected. Moreover, since $G$ is a chordless graph $B$ is minimally $2$-connected. Let $x \in V(B)$ be the cut vertex whose removal separates $B$ from rest of the graph. By Lemma \ref{degreeLemma}, $\exists u,v,w \in V(B)$ such that $v,w \in N_B(u)$, $x \notin \{v,w\}$, and $deg_B(v) = deg_B(w) = 2$. Since neither $v$ nor $w$ is a cut vertex, their degrees in $G$ remain the same, i.e., $deg_G(v) = deg_G(w) = 2$. Now, let $S = \{u,v,w\}$. By induction hypothesis, $\lambda_2(G \setminus S) \geq \frac{2(n-3)}{3}$. %Consider a partial list colouring of $G \setminus S$ us 
Consider adding $u,v$ and $w$ back in  $G \setminus S$. It is easy to see that if we don't colour vertex $u$, then both $v$ and $w$ have at most one of their neighbours coloured and therefore can be properly coloured using a colour from their respective lists. Thus, $\lambda_2(G) \geq \frac{2(n-3)}{3} + 2 = \frac{2n}{3}$. 
\end{proof} 

\subsection{Series-parallel graphs}
\label{Ser-Par-SubSec}
A connected \emph{series-parallel} graph is a graph with two designated vertices $s$ and $t$ and the graph can be turned into a $K_2$  by a sequence of the following operations: (a) Replacement of a pair of edges with a single edge that connects their endpoints, and (b) Replacement of a pair of edges incident to a vertex of degree $2$ other than $s$ or $t$ with a single edge. Note that outerplanar graphs are series-parallel. From the definition, it is easy to see that series-parallel graphs are $2$-degenerate and therefore $3$-choosable. Hence, in order to prove the partial list colouring conjecture for a series-parallel graph $G$ on $n$ vertices, it is enough to show that $\lambda_2(G) \geq \frac{2n}{3}$. We prove this by using the fact that series-parallel graphs are precisely the class of graphs with treewidth at most $2$.  Before we prove this result, let us explore the connection between the treewidth of a graph $G$ and $\lambda_t(G)$. 
\begin{proposition}
\label{chromaticProp}
Let $\mathcal{G}$ be a hereditary graph family. If  $\forall G \in \mathcal{G}$, $\chi(G) = \chi_L(G)$ then  for every $t$, where $0 < t \leq  \chi_L(G)$, $\lambda_t(G) \geq \frac{tn}{\chi_L(G)}$. 
\end{proposition}
\begin{proof}
The proof is straightforward. Consider a graph $G \in \mathcal{G}$. Let $\chi_L(G) = \chi(G) = s$. Remove the vertices belonging to the smallest $s-t$ colour classes in a proper $s$-colouring of $G$ to obtain a graph $G'$ of size at least $\frac{tn}{s}$. Since $G'\in \mathcal{G}$, $\chi_L(G') = \chi(G') = t$.   
\end{proof}
A graph is \emph{chordal} if it does not contain any induced cycle of size greater than $3$. Note that chordal graphs form a hereditary family of graphs. Moreover, for every chordal graph $G$, $\chi(G)=\chi_L(G)=\omega(G)$ (clique number). Since chordal graphs form a hereditary family of graphs satisfying the condition given in  Proposition \ref{chromaticProp}, we have the following corollary.  
\begin{corollary}
\label{chordalCorollary}
Let $G$ be a chordal graph on $n$ vertices whose list chromatic number is $s$. Then for every positive integer $t$ less than $s$, $\lambda_t(G) \geq \frac{tn}{s}$. 
\end{corollary}
We know that the treewidth of a graph $G$, denoted by $tw(G)$, is one less than the largest clique in the chordal graph containing $G$ with the smallest clique number. The corollary below relates $\lambda_t(G)$ with the treewidth of $G$. 
\begin{corollary}
\label{twCorollary}
Consider a graph $G$ on $n$ vertices whose list chromatic number is $s$. For every positive integer $t$ less than $s$, $\lambda_t(G) \geq \frac{tn}{tw(G) + 1}$. 
\end{corollary}
\begin{proof}
Let $G'$ be a chordal graph obtained by adding edges to $G$ such that the size of a largest clique in $G'$, denoted by $\omega(G')$, is $tw(G)+1$.  Since $tw(G)+1 = \omega(G')=\chi(G') = \chi_L(G')$, from Corollary \ref{chordalCorollary},  we get $\lambda_t(G') \geq \frac{tn}{tw(G) + 1}$. As $G$ is a subgraph of $G'$, any proper colouring of $G'$ is a proper colouring for $G$ as well. Thus we get $\lambda_t(G) \geq \frac{tn}{tw(G) + 1}$. 
\end{proof}
The following result on series-parallel graphs follows directly from Corollary \ref{twCorollary}. 
\begin{corollary}
Let $G$ be a series-parallel graph on $n$ vertices whose list chromatic number is $s$. Then for every positive integer $t$ less than $s$, $\lambda_t(G) \geq \frac{tn}{s}$. 
\end{corollary}
\section{$2$-choosable graphs in $3$-choosable graphs}
\label{StrongerConjectureSection}
In this section, we put forth a question which is a (stronger) variant of the partial list colouring conjecture.  
Given a graph $G$ on $n$ vertices with list chromatic number $s$, for any positive integer $t<s$, does $G$ contain a $t$-choosable graph of size $\frac{tn}{s}$?  For instance, observe that the proof of Theorem \ref{largeChromatThm} in fact proves this stronger statement to be true in the context of graphs with large chromatic number. That is, every $s$-choosable graph $G$ with chromatic number at least $\frac{n-1}{2}$ contains a subgraph induced on $\frac{tn}{s}$ vertices that is $t$-choosable. 
%Series-parallel graphs are another example where this stronger statement holds true. 
It is easy to see that the arguments outlined in Section \ref{Ser-Par-SubSec} can be used to state that the stronger statement holds true for series-parallel graphs as well. In what follows, we show that the answer to this question is not always `yes' by explicitly constructing an infinite family of $3$-choosable graphs where a largest induced $2$-choosable subgraph of each graph in the family is of size at most $\frac{5n}{8}$.

Given a graph $G$ with its vertices arranged in the order $v_1, \ldots , v_n$, a list assignment $\mathcal{L} = \{l(v_i)~:~i \in [n]\}$ is called an $(l_1, \ldots , l_n)$-assignment if $\forall i \in [n],~|l(v_i)|= l_i$. 
We say $G$ is $(l_1, \ldots , l_n)$-list colourable if $G$ is $\mathcal{L}$-list colourable for every $(l_1, \ldots , l_n)$-assignment $\mathcal{L}$.  

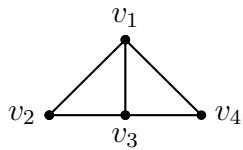
\begin{figure}[!ht]
\begin{center}
\begin{pspicture}(-0.5,-0.5)(2.5,1.5)
% 	\psgrid
	\psline[showpoints=true](0,0)(1,0)(1,1)(0,0)
	\psline[showpoints=true](1,0)(2,0)(1,1)

	\uput[u](1,1){$v_1$}
	\uput[l](0,0){$v_2$}
	\uput[270](1,0){$v_3$}
	\uput[r](2,0){$v_4$}
	
\end{pspicture}
\end{center}
\caption{Diamond graph $D$}
\label{figureDiamond}
\end{figure}

\begin{proposition}
\label{DiamondClaim}
Let $D$ denote the diamond graph (see Figure \ref{figureDiamond}) whose vertices are arranged in the order $v_1, \ldots , v_4$. Then, $D$ is $(2,2,3,2)$-list colourable.  
\end{proposition}
\begin{proof}
Let $\mathcal{L} = \{l(v_i)~:~1 \leq i \leq 4\}$ be an arbitrary $(2,2,3,2)$-assignment for $D$.  
%We shall prove the claim by constructing a proper list colouring $f:v \rightarrow l(v)$. 
If $l(v_2) \cap l(v_4) \neq \emptyset$, then   let us assign a colour $x$ to both $v_2$ and $v_4$, where $x \in  l(v_2) \cap l(v_4)$. It's easy to colour $v_1$ followed by $v_3$ as their respective list sizes  are now one more than the number of colours in their neighbourhood. If $l(v_2) \cap l(v_4) =  \emptyset$, then one can first colour $v_1$ followed by $v_3$ in such a way that both $v_2$ and $v_4$ have a colour in their respective lists unused by any of their neighbours. 
\end{proof}

\begin{figure}[!ht]
\begin{center}
\begin{pspicture}(-0.5,-0.5)(12.5,3.5)
% 	\psgrid
	\psline[showpoints=true](1,0)(0,1)(0,2)(1,3)(2,2)(2,1)(1,0)
	\psline[showpoints=true](1,0)(1,1)(1,2)(1,3)
	\psline[showpoints=true](0,1)(1,1)(2,1)
	\psline[showpoints=true](0,2)(1,2)(2,2)
	\psline[showpoints=true](0,2)(2,1)
%	\pscurve(0,2)(1,4)(3,2)(2,1)
	
	\psline[showpoints=true](5,0)(4,1)(4,2)(5,3)(6,2)(6,1)(5,0)
	\psline[showpoints=true](5,0)(5,1)(5,2)(5,3)
	\psline[showpoints=true](4,1)(5,1)(6,1)
	\psline[showpoints=true](4,2)(5,2)(6,2)
	\psline[showpoints=true](4,2)(6,1)

	\psline[showpoints=true](1,0)(5,0)

	\psline[linestyle=dotted](7,1.5)(9,1.5)

	\psline[showpoints=true](11,0)(10,1)(10,2)(11,3)(12,2)(12,1)(11,0)
	\psline[showpoints=true](11,0)(11,1)(11,2)(11,3)
	\psline[showpoints=true](10,1)(11,1)(12,1)
	\psline[showpoints=true](10,2)(11,2)(12,2)
	\psline[showpoints=true](10,2)(12,1)

	\psline[showpoints=true](1,0)(5,0)
	\psline(11,0)(9,0)
	\uput[u](1,3){$v_{1,1}$}
	\uput[u](5,3){$v_{2,1}$}
	\uput[u](11,3){$v_{r,1}$}
	
	\uput[l](0,2){$v_{1,2}$}
	\uput[-45](1,2){$v_{1,3}$}
	\uput[r](2,2){$v_{1,4}$}
	\uput[l](4,2){$v_{2,2}$}
	\uput[-45](5,2){$v_{2,3}$}
	\uput[r](6,2){$v_{2,4}$}
	\uput[l](10,2){$v_{r,2}$}
	\uput[-45](11,2){$v_{r,3}$}
	\uput[r](12,2){$v_{r,4}$}

	\uput[l](0,1){$v_{1,5}$}
	\uput[135](1,1){$v_{1,6}$}
	\uput[r](2,1){$v_{1,7}$}
	\uput[l](4,1){$v_{2,5}$}
	\uput[135](5,1){$v_{2,6}$}
	\uput[r](6,1){$v_{2,7}$}
	\uput[l](10,1){$v_{r,5}$}
	\uput[135](11,1){$v_{r,6}$}
	\uput[r](12,1){$v_{r,7}$}

	\uput[270](1,0){$v_{1,8}$}
	\uput[270](5,0){$v_{2,8}$}
	\uput[270](11,0){$v_{r,8}$}

\end{pspicture}
\end{center}
\caption{Graph $H$}
\label{figureH}
\end{figure}

\begin{lemma}
\label{3choosableLemma}
Let $H$ denote the graph in Figure \ref{figureH}. Then, $H$ is $3$-choosable. 
\end{lemma}
\begin{proof}
%Let $v_{1,1}, \ldots , v_{1,8}, \ldots , v_{r,1}, \ldots , v_{r,8}$ be the vertices of $H$ arranged in that order. 
Let $H_i$ denote a subgraph of $H$ induced on the vertex set $\{v_{i,1}, \ldots , v_{i,8}\}$. In order to prove  the lemma, it is enough to  prove that $H_i$ is $(3,3,3,3,3,3,3,2)$-list colourbale with respect to the ordering $v_{i,1}, \ldots , v_{i,8}$ of its vertices. For ease of notation, let us relabel each $v_{i,j}$ in $H_i$ by $v_j$ for every $1 \leq j  \leq 8$. Let $\mathcal{L} = \{l(v_{j})~:~1 \leq j \leq 8\}$ be a $(3,3,3,3,3,3,3,2)$-assignment for $H_i$. We split the proof into two cases:  
\\ \textit{Case} $l(v_2) \cap l(v_4) \neq \emptyset$: Assign both $v_2$ and $v_4$ a colour $x$, where $x \in l(v_2) \cap l(v_4)$. We can then say that $H_i$ is $\mathcal{L}$-list colourable if its subgraph induced on vertices $v_1, v_3, v_5, v_6, v_7,v_8$ in that order is $(2,2,2,3,2,2)$-list colourable. From Proposition \ref{DiamondClaim}, we know that the graph induced on vertices $v_5, v_6, v_7,$ and $v_8$ can be properly list coloured using their updated lists. It's easy to colour $v_3$ followed by $v_1$ as their respective list sizes  are now one more than the number of colours in their neighbourhood. 
\\ \textit{Case} $l(v_5) \cap l(v_7) \neq \emptyset$: Assign both $v_5$ and $v_7$ a colour $x$, where $x \in l(v_5) \cap l(v_7)$. We can then say that $H_i$ is $\mathcal{L}$-list colourable if its subgraph $D'$ induced  on vertices $v_1, v_2, v_3, v_4, v_6, v_8$ in that order is $(3,2,3,2,2,1)$-list colourable. In $D'$, first colour $v_8$ followed by $v_6$ as their respective list sizes are now one more than the number of colours in their neighbourhood. Note that $D'$ is $(3,2,3,2,2,1)$-list colourable if the diamond graph induced on vertices $v_1, \ldots , v_4$ in that order is $(3,2,2,2)$-list colourable. 
%Using Claim \ref{DiamondClaim}, we can conclude that $H_i$ is $\mathcal{L}$-list colourable.    
From Proposition \ref{DiamondClaim}, we know that this diamond graph can be properly list coloured using its updated lists. 
\\ \textit{Case} $l(v_2) \cap l(v_4) = \emptyset$ and $l(v_5) \cap l(v_7) = \emptyset$: Here it's easy to see that either $|l(v_2) \cap l(v_5)| \leq 1$ or $|l(v_2) \cap l(v_7)| \leq 1$. Assume $|l(v_2) \cap l(v_5)| \leq 1$ (the proof is similar when $|l(v_2) \cap l(v_7)| \leq 1$). Let $l'(v_i) = l(v_i)$ if $i \neq 5$ and let $l'(v_5) = l(v_5) \setminus (l(v_2) \cap l(v_5))$. Let $\mathcal{L}' = \{l'(v_i)~:~1 \leq i \leq 8\}$. Clearly, if $H_i$ is $\mathcal{L}'$-list colourable, then it is $\mathcal{L}$-list colourable. By Proposition \ref{DiamondClaim}, we can say that the subgraph of $H_i$ induced on vertices $v_5,v_6,v_7,$ and $v_8$ is  $\mathcal{L}'$-list colourable. Once these vertices are coloured, each of $v_2$, $v_3$, and $v_4$ have at least two unused colours in their respective lists, that is,  at least two colours that have not been used by any of their respective neighbours so far. Also, $v_1$ has $3$ unused colours in its list. Thus, $H_i$ is $\mathcal{L}'$-list colourable if the subgraph  induced on vertices $v_1, v_2, v_3,$ and  $v_4$ in that order is $(3,2,2,2)$-list colourable. Thus it follows from Proposition \ref{DiamondClaim} that $H_i$ is $\mathcal{L}'$-list colourable.  
\end{proof}

Before we prove the main theorem of this section, let us recollect the famous characterization of $2$-choosable graphs due to Erd{\"o}s et al. in \cite{erdos1979choosability}. For a positive integer $k$, let $\theta(2,2,2k)$ denote the graph with two designated vertices $u$ and $v$ and three vertex disjoint paths between them, where each path is of length $2$, $2$, and $2k$ respectively. Given a connected graph $G$, let \emph{core($G$)} denote the graph obtained from $G$ by successive deletion of all vertices of degree $1$. Then, $G$ is $2$-choosable if and only if core($G$) is one of the three: an isolated vertex $K_1$, an even cycle $C_{2k}$, or a $\theta(2,2,2k)$, where $k$ is any positive integer.  
\begin{theorem}
%Let $H$ denote the $3$-choosable graph in Figure \ref{}. Then, the largest induced $2$-choosable subgraph of $H$ has at most $\frac{5n}{8}$ vertices. 
%There exists a $3$-choosable graph of size $n$ whose largest induced $2$-choosable subgraph is of size at most $\frac{5n}{8}$. 
Let $r$ be any positive integer and let $n = 8r$. Then, there exists a $3$-choosable graph $G$ on $n$ vertices such that the size of its largest induced $2$-choosable subgraph is $\frac{5n}{8}$.  
\end{theorem}
\begin{proof}
Consider the graph $H$ in Figure \ref{figureH}. It has $n$ vertices, where $n=8r$, for some positive integer $r$. By Lemma \ref{3choosableLemma}, $H$ is $3$-choosable. In order to prove the theorem, it is enough to show that the largest $2$-choosable graph in any $H_i$ has at most $5$ vertices, where $H_i$ is the subgraph of $H$ induced on the vertex set $\{v_{i,1}, \ldots ,  v_{i,8}\}$. Assume for contradiction that $H'$ is an induced subgraph of $H_i$ such that $|V(H')|=6$ and $\chi_L(H') = 2$. Clearly, $V(H_i) \setminus V(H')$ should be a hitting set for all the four triangles in $H_i$. That is, $V(H_i) \setminus V(H')$ is one of $\{v_{i,1}, v_{i,6}\},~\{v_{i,1}, v_{i,8}\},~\{v_{i,3}, v_{i,6}\},$ or $\{v_{i,3}, v_{i,8}\}$.  In all the above cases $H'$ is a super graph of $G_8$ (see Figure \ref{figureG8}), where $G_8$ is the graph obtained by pasting two $4$-cycles on an edge. Since core($G_8$) is none of $K_1$, $C_{2k}$, or $\theta(2,2,2k)$, $G_8$ is not $2$-choosable. Therefore, $H'$ is not $2$-choosable.  

Consider the subgraph of $H$ induced on the vertex set $\{v_{i,j}~:~1 \leq i \leq r,~3 \leq j \leq 7\}$. This subgraph is $2$-choosable as the core of each of its connected components is a four cycle. Thus, we get a $2$-choosable subgraph of $H$ of size $\frac{5n}{8}$. 
\end{proof}

We shall now show that though the graph $H$ in Figure \ref{figureH} disproves the stronger variant of the partial list colouring conjecture stated in the beginning of this section by the authors, it is still no counterexample to the partial list colouring conjecture of \cite{albertson2000partial}. Below we prove a lemma which will aid us in showing this. 

\begin{figure}[!ht]
\begin{center}
\begin{pspicture}(-0.5,-0.5)(2.5,1.5)
% 	\psgrid
	\psline[showpoints=true](0,0)(1,0)(2,0)(2,1)(1,1)(0,1)(0,0)
	\psline[showpoints=true](1,0)(1,1)

	\uput[l](0,1){$v_1$}
	\uput[u](1,1){$v_2$}
	\uput[r](2,1){$v_3$}
	\uput[l](0,0){$v_4$}
	\uput[270](1,0){$v_5$}
	\uput[r](2,0){$v_6$}
	
\end{pspicture}
\end{center}
\caption{Graph $G_8$}
\label{figureG8}
\end{figure}
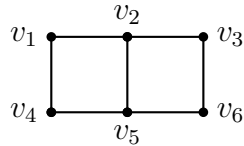

\begin{lemma}
\label{G8Lemma}
Let $\mathcal{L}=\{l(v_i)~:~1 \leq i \leq 8\}$ be a $2$-assignment for the graph $G_8$ of Figure \ref{figureG8} such that $l(v_1)=l(v_4)$ (or equivalently $l(v_3)=l(v_6)$). Then $G_8$ is $\mathcal{L}$-list colourable. 
\end{lemma}
\begin{proof}
We first list colour the $4$-cycle induced on vertices $v_2, v_3,v_5$, and $v_6$ as we know that every $4$-cycle is $2$-choosable. Let $c_2$ and $c_5$ be the colours assigned to $v_2$ and $v_5$, respectively, in this colouring. If $c_2 \notin l(v_1)$ (or $c_5 \notin l(v_4)$), then we can easily list colour both $v_1$ and $v_4$. Suppose, $c_2 \in l(v_1)$ and $c_5 \in l(v_4)$. Since $l(v_1) = l(v_4)$, we have $l(v_1) = l(v_4) = \{c_2, c_5\}$. Assigning colours $c_5$ and $c_2$ to $v_1$ and $v_4$, respectively, gives a valid list colouring for $G_8$.    
\end{proof}

\begin{theorem}
The partial list colouring conjecture holds true for the graph $H$ in Figure \ref{figureH}.
\end{theorem}
\begin{proof}
Let $n$ denote the number of vertices of $H$, where $n=8r$. From Lemma \ref{3choosableLemma}, we know that $H$ is $3$-choosable. Hence to prove the theorem it's enough to show that $\lambda_2(H) \geq \frac{2n}{3}$. Let $H_i$ denote the subgraph of $H$ induced on the vertex set $\{v_{i,j}~:~1 \leq j \leq 8\}$. We prove the theorem by showing that given any $2$-assignment for $H$, for every odd integer $i$ ranging from $1$ to $r$, and for every even integer $j$ ranging from $1$ to $r$, we can properly list colour $6$ vertices of every $H_i$ and  $5$ vertices of every $H_{j}$. Since the subgraph induced on vertices $v_{j,3}, v_{j,4}, v_{j,5}, v_{j,6}$, and $v_{j,7}$ is $2$-choosable (see characterization of $2$-choosable graphs in \cite{erdos1979choosability}), every $H_j$ contains $5$ vertices that can be properly list coloured using the given $2$-assignment for $H$. Moreover, note that none of these coloured vertices have a neighbour in any $H_{k}$, where $k \neq j$ and $1 \leq k \leq r$. Now, consider any $H_i$, where $i$ is even. Suppose the lists assigned to vertices $v_{i,1}, v_{i,3}$, and $v_{i,4}$ are not all the same then we can colour the triangle induced on them. Moreover, we can also colour the vertices $v_{i,7}, v_{i,8}$, and $v_{i,5}$ in that order as each of them see at most $1$ colour in its neighbourhood when it is coloured. Hence, in this case we can colour $6$ vertices from $H_i$. Consider the case when  vertices $v_{i,1}, v_{i,3}$, and $v_{i,4}$ are all assigned the same list. In this case, consider the subgraph $H_i'$ induced on vertices $v_{i,1}, v_{i,4}, v_{i,7}, v_{i,8}, v_{i,5}$, and $v_{i,2}$. Note that $H_i'$ is isomorphic to $G_8$. Moreover, both $v_{i,1}$ and $v_{i,4}$ have been assigned the same list. Then, by Lemma \ref{G8Lemma}, we know that $H_i'$ is $2$-choosable. Hence, we prove the theorem.

\end{proof}
\bibliographystyle{plain}
%\bibliography{mathewref}

\end{document}